\documentclass[twoside, 12pt]{article}
\usepackage{mathrsfs}
\usepackage{amssymb, amsmath, mathrsfs, amsthm}
\usepackage{graphicx}
\usepackage{subcaption}
\usepackage{color}
\usepackage{tikz}
\usepackage[top=2cm, bottom=2cm, left=1.8cm, right=1.8cm]{geometry}
\usepackage{float, caption, subcaption}
\usepackage{diagbox}
\usepackage{algorithm}
\usepackage{algpseudocode}
\usepackage{amsmath}
\usepackage[colorlinks,
  linkcolor=blue,
 anchorcolor=blue,
 citecolor=blue]{hyperref} 
\input{epsf}

\newcommand{\bd}{\begin{description}}
\newcommand{\ed}{\end{description}}
\newcommand{\bi}{\begin{itemize}}
\newcommand{\ei}{\end{itemize}}
\newcommand{\be}{\begin{enumerate}}
\newcommand{\ee}{\end{enumerate}}
\newcommand{\beq}{\begin{equation}}
\newcommand{\eeq}{\end{equation}}
\newcommand{\beqs}{\begin{eqnarray*}}
\newcommand{\eeqs}{\end{eqnarray*}}

\definecolor{DarkGreen}{rgb}{0.2, 0.6, 0.3}

\usepackage{multicol}
\usepackage{amsthm}
\usepackage{multirow}
\usepackage{makecell}

\newtheorem{theorem}{Theorem}[section]
\newtheorem{conjecture}{Conjecture}

\newtheorem{lemma}{Lemma}[section]
\newtheorem{definition}{Definition}[section]

\newtheorem{case}{Case}

\newtheorem{claim}{Claim}

\begin{document}
\title{\textbf{Induced Ramsey numbers for fans} \footnote{Supported by the National Science Foundation of China
(Nos. 12471329 and 12061059) and Qinglan Project of Jiangsu Province of China.}
}

\author{Chuang Zhong\footnote{School of Mathematics and Statistis, Qinghai Normal University, Xining, Qinghai 810008, China. {\tt zhongchuang1998@163.com}}, \ \ Masaki Kashima\footnote{Faculty of Science and Technology, Keio University, Yokohama, Kanagawa, Japan. {\tt  masaki.kashima10@gmail.com}}, \ \
Yaping Mao\footnote{Academy of Plateau
Science and Sustainability, and School of Mathematics and Statistis, Qinghai Normal University, Xining, Qinghai 810008,
China. {\tt yapingmao@outlook.com; myp@qhnu.edu.cn}}, \ \ Yan Zhao \footnote{Corresponding author: Department of Mathematics, Taizhou University, Taizhou 225300, China. {\tt zhaoyan326@tzu.edu.cn}}
}
\date{}
\maketitle

\begin{abstract}
The induced Ramsey number $r_{\mathrm{ind}}(G,H)$ is defined as the minimum order of a graph $F$ on such that any 2-coloring of its edges with red and blue leads to either a red induced copy of $G$ or a blue induced copy of $H$. Motivated by the Kohayakawa-Prömel-Rödl conjecture, we prove that a quadratic upper bound $\mathrm{r}_{\text {ind}}\left(G, F_n\right) \leq C n^2$ for fixed $G$, where $F_n$ is a graph with one central vertex, $2n$ leaf vertices, and $n$ disjoint edges. In particular, for star graphs $K_{1, \ell}$ $(\ell \leq n)$, constructive coloring and matching arguments yield $2 n+2 \ell-1 \leq \mathrm{r}_{\text {ind}}\left(K_{1, \ell}, F_n\right) \leq(\ell+n-1)(\ell+1)+1$, with the exact value $\mathrm{r}_{\text {ind}}\left(K_{1,2}, F_n\right)=3 n+4$. \\[2mm]

\end{abstract}

\section{Introduction}

All graphs in this paper are undirected, finite and simple, and any undefined concepts or notation can be found in \cite{BM}. Let $G=(V(G),E(G))$ be a graph with vertex set $V(G)$ and edge set $E(G)$.  

We write $F\xrightarrow{ind}(G,H)$ for graphs $F$, $G$, and $H$, if for any coloring of the edges of $F$ in
red and blue, there is either a red induced copy of $H$ or a blue induced copy of $G$. A red induced copy of a graph $H$ (resp. blue induced copy of $H$) is an induced subgraph $F'$ of $F$ such that $F' \cong H$ and all edges of $F'$ are red (resp. blue).
\begin{definition}
For
graphs $G$ and $H$, the induced Ramsey number for $H$ and $G$, denoted by $r_{\mathrm{ind}}(G,H)$, is the
smallest number of vertices in a graph $F$ such that $F\xrightarrow{ind}(G,H)$.   
\end{definition}

\begin{definition}
For graphs $G$ and $H$, the \textit{Ramsey number} $r(G,H)$ is the minimum integer $n$ such that any red/blue coloring of the edges of the complete graph  $K_{n}$ contains a red copy of $G$ or a blue copy of $H$ as a subgraph.
\end{definition}

The existence of such a graph $F$ for any graphs $H$ and $G$ was proven by Deuber \cite{Deuber75}, Erd\H{o}s, Hajnal, and Pósa \cite{EHP75}, and R\"{o}dl \cite{R73,R76}. This led to extensive research on induced Ramsey numbers. Since in the case of complete graphs the induced subgraph is the same as the subgraph, it is obvious that $r_{\mathrm{ind}}(K_{m},K_{n}) = r(K_{m},K_{n})$. Erd\H{o}s \cite{E75} conjectured that there is a positive constant $c$ such that each graph $G$ with $n$ vertices satisfies $r_{\mathrm{ind}}(G,G) \leq 2^{cn}$. 
R\"{o}dl \cite{R73} proved that $r_{\mathrm{ind}}(G,G)$ is indeed exponential in $n=|V (G)|$ if $G$ is bipartite. 
A recent paper by Kohayakawa, Pr\"{o}mel, and R\"{o}dl \cite{KPR98} established $r_{\mathrm{ind}}(G,G) \leq n^{cn\log \chi(G)}$, where $\chi(G)$ is the chromatic number of $G$.
In particular, their result implies an upper bound of $2^{cn(\log n)^2}$ that give a positive answer to the conjecture of Erd\H{o}s. Further, they proved that if $G$ belongs to certain restricted classes of graphs, the upper bound becomes polynomial in $n$.
In their proof, the graphs which give the bounds are randomly constructed using projective planes. Subsequently, Conlon, Fox, and Sudakov \cite{CFS12} showed that $r_{\mathrm{ind}}(G,G) \leq 2^{cn \log n}$, improving the previous bound. 
Haxell, Kohayakawa,
and Łuczak \cite{HKL95} establishd that the cycle of length $k$ has induced Ramsey number linear in $k$. Subsequently, Fox and Sudakov \cite{FS08} introduced a unified approach to proving Ramsey‐type theorems for graphs
with a forbidden induced subgraph. This method not only provides explicit constructions for upper bounds on various induced Ramsey numbers but also yields further linearity results. They demonstrated that, although the induced Ramsey number is linear in $k$ for a path or a star on $k$ vertices, it becomes superlinear for any tree that contains both a $k$-vertex path and a $k$-vertex star.

Further results are obtained when $G$ and $H$ are fixed.
For example, 
Gorgol \cite{Gorgol01} gave the exact value of $r_{\mathrm{ind}}(K_{1,m}, K_n)$. 
Gorgol and {\L}uczak \cite{GL02}
showed $r_{\mathrm{ind}}(kK_2,K_n)=kn$.
Schaefer and Shah \cite{SS03} provided upper bounds of $r_{\mathrm{ind}}(C_n,C_n)$, $r_{\mathrm{ind}}(T,K_n)$ and $r_{\mathrm{ind}}(B,C_n)$ by constructive methods.
Kostochka and Sheikh \cite{KS06}
gave the upper bound of $r_{\mathrm{ind}}(P_3,H)$, and proved that
$r_{\mathrm{ind}}(P_3, K_{n_1} + K_{n_2} + \ldots + K_{n_m}) = \sum_{i=1}^m \frac{n_i(n_i+1)}{2}$ for any positive integers $n_1 \leq \ldots \leq n_m$.
For more results, see \cite{AG20star,AG20,Gorgol19}.

In 1998, Kohayakawa, Pr\"omel and R\"odl proposed the following conjecture. 
\begin{conjecture}{\upshape \cite{KPR98}}\label{con-1}
For any graph $G$, there is a constant $f=f(G)$ that depends only on $G$ such that, for any graph $H$ on $n$ vertices, we have
$$
r_{\mathrm{ind}}(G, H) \leq n^f.
$$    
\end{conjecture}

Towards this conjecture, they obtained a general upper bound that applies to all graphs, albeit with an exponent heavily dependent on the chromatic number of the target graph.

\begin{theorem}{\upshape \cite{KPR98}}\label{KPR1}
Let $G$ and $H$ be graphs with $|V(G)|=k$ and $|V(H)|=n$, where $k \leq n$. If $q=\chi(H) \geq 2$, then
$$
r_{\mathrm{ind}}(G, H) \leq n^{C k \log q}
$$  
for some absolute constant $C$.
\end{theorem}

In the same paper, they completely verified Conjecture \ref{con-1} for a restricted family known as "simple graphs".
Let $\mathcal{S}$ denote the smallest class of finite graphs that contains the 1-vertex complete graph and is closed under taking disjoint unions and joins. Following Erd\"{o}s and Hajnal \cite{EH89}, we call the graphs in $\mathcal{S}$ simple.

\begin{theorem}{\upshape \cite{KPR98}}\label{KPR}
For any simple graph \( G \), there is a constant \( f = f(G) \) that depends only on \( G \) such that, for any graph \( H \) on \( n \) vertices, we have  
\[
r_{\mathrm{ind}}(G, H) \leq n^f.
\]
\end{theorem}
A natural question arises regarding the precise asymptotic behavior of the induced Ramsey number when the target graph has a highly structured, star-like topology but grows arbitrarily large. To this end, we study the \textit{fan graph} $F_n$, defined as a graph with one central vertex, $2n$ leaf vertices, and $n$ disjoint edges (called matching edges) that pair the leaves. Each matching edge forms a triangle with the central vertex. Finding an induced copy of $F_n$ is notoriously rigid, as it must preserve all these triangular edges while ensuring no other edges exist between distinct triangular leaves.

Motivated by Conjecture \ref{con-1}, we 
use the probabilistic methods (random projective plane constructions) and the Caro-Wei lemma to prove the following quadratic upper bound for fixed $G$:
\begin{theorem}\label{upperbound}
Let $G$ be a fixed graph. There exists a constant $C = C(G)$, depending only on $G$, such that
\[
r_{\mathrm{ind}}(G,F_n) \le C n^2.
\]
\end{theorem}

It is worth noting the fundamental distinction between our Theorem \ref{upperbound} and the classical bounds. If one were to apply Theorem \ref{KPR} to evaluate $r_{\mathrm{ind}}(G, F_n)$, the resulting bound would be $O((2n+1)^{f(G)})$. The exponent $f(G)$ is a recursively defined constant computed along the decomposition tree of $G$. Consequently, as the structure of the fixed graph $G$ becomes more complex, this recursive exponent inflates rapidly, resulting in a polynomial bound of an excessively high degree. In stark contrast, our Theorem \ref{upperbound} completely bypasses this recursive inflation. By exploiting the geometric properties of the projective plane, we eliminate the dependence of the polynomial's degree on $G$. We successfully restrict the impact of $G$ entirely to the leading coefficient $C(G)$, collapsing the upper bound to a tight, absolute quadratic order of $O(n^2)$.

While Theorem \ref{upperbound} establishes a robust universal asymptotic bound for any fixed $G$, determining the exact values or tight linear bounds for specific sparse graphs remains a mathematically challenging endeavor. To sharpen these bounds further for specific graph families, we restrict our attention to the case where $G$ is a star graph $K_{1,\ell}$. Since $F_1\simeq K_3$, it follows from a result by Gorgol~\cite{Gorgol01} that $r_{\mathrm{ind}}(K_{1,\ell},F_1)=r_{\mathrm{ind}}(K_{1,\ell},K_3)=3\ell$.
For general fan graph $F_n$, using delicate constructive coloring and matching arguments, we obtain the following tight bounds, which are exact for $\ell=2$:
\begin{theorem}\label{star-fan-bound}
Let $n$, $\ell$ be positive integers with $\ell\leq n$. Then
$$
2n+2\ell-1\leq r_{\mathrm{ind}}(K_{1,\ell},F_{n})\leq (\ell+n-1)(\ell+1)+1.
$$
Furthermore, if $\ell=2$, then 
$r_{\mathrm{ind}}(K_{1,2},F_{n})=3n+4$, which implies that the upper bound is sharp.
\end{theorem}

\section{Preliminaries}

Let $G$ be an undirected graph. For a vertex $v \in V(G)$, its neighborhood $N_{G}(v)$ is the set of vertices adjacent to $v$, and its closed neighborhood is $N_{G}[v] = N_{G}(v) \cup \{v\}$. The degree of a vertex $v \in V(G)$ is $d_{G}(v) = |N_{G}(v)|$. We denote the minimum degree and maximum degree of $G$ by $\delta(G)$ and $\Delta(G)$, respectively, where $\delta(G) = \min\limits_{v \in V(G)}\{d_{G}(v)\}$ and $\Delta(G) = \max\limits_{v \in V(G)}\{d_{G}(v)\}$. 
For a set of vertices $S\subseteq V(G)$, $G[S]$ is the subgraph of $G$ induced by $S$.
For two disjoint subset $S$ and $T$ of $V(G)$, let $E_G(S,T)$ denote the set of edges of $G$ with one end in $S$ and the other end in $T$, and let $e_G(S,T)=|E_G(S,T)|$.
In particular, we write $E_G(v,T)$ or $e_G(v,T)$ instead of $E_G(\{v\},T)$ or $e_G(\{v\},T)$ for a vertex $v$.

The \textit{clique number} of $G$, denoted by $\omega (G)$, is the size of the largest clique in $G$.
We denote the \textit{independence number} of $G$ by $\alpha(G)$, which is the size of the largest independent set in $G$.
Note that $\omega(G) = \alpha(\overline{G})$, where $\overline{G}$ is the complement of $G$.
The \textit{chromatic number} of $G$, denoted by $\chi(G)$, is the minimum number of colors needed to color the vertices of $G$ such that no two adjacent vertices share the same color.

We denote by $K_n$ the complete graph on $n$ vertices, by $K_{1,\ell}$ the star graph with one central vertex and $\ell$ leaf vertices. 
A graph is called \textit{triangle-free} if it contains no $K_3$ (triangle) as a subgraph.
Undefined graph-theoretical terms and notation follow \cite{BM}.

A finite projective plane \cite{Hirschfeld1998} of order $p$ is an incidence structure
$\mathcal P=(V,\mathcal L)$ satisfying the following properties:
\begin{itemize}
\item $|V|=p^2+p+1$ and $|\mathcal L|=p^2+p+1$;
\item every line $L\in\mathcal L$ contains exactly $p+1$ vertices;
\item every vertex $v\in V$ lies on exactly $p+1$ lines;
\item any two distinct lines intersect in exactly one vertex.
\item Any two distinct vertices determine a unique line.
\end{itemize}
It is well known that finite projective planes exist whenever $p$ is a prime power.

The following lower bound of the independence number of a graph was given by Caro and Wei~\cite{AS16}.

\begin{lemma}{\upshape \cite{AS16}}\label{Caro–Wei}
    Every graph $G$ contains an independent set of size at least
    $$\alpha(G)\geq\sum_{v\in V(G)}\frac{1}{d_{G}(v)+1}.$$
\end{lemma}

\section{General upper and lower bounds}

We study the induced Ramsey number $r_{\mathrm{ind}}(H, F_n)$ where $F_n$ is the fan graph consisting of $n$ triangles sharing a common vertex. The proof uses the random graphs based on projective planes and probabilistic methods.

\newtheorem*{mainthm1}{\rm\bf Theorem~\ref{upperbound}}
\begin{mainthm1}[Restated]
Let $G$ be a fixed graph. There exists a constant $C = C(G)$, depending only on $G$, such that
\[
r_{\mathrm{ind}}(G,F_n) \le C n^2.
\]
\end{mainthm1}

\begin{proof}
By the induced Ramsey theorem, there exists a finite graph $G^*$ such that
\[
G^* \xrightarrow{\mathrm{ind}} (G,K_3).
\]
Let $t^* = |V(G^*)|$. We choose constants satisfying the following conditions:
\begin{enumerate}
    \item Let $\Delta = 2t^*$;
    \item Choose $C_2$ sufficiently large such that
    \[
    \frac{9C_2}{4(\Delta + 1)} \ge 1;
    \]
    \item Choose $C_1$ sufficiently large relative to $C_2$ and $t^*$ to satisfy 
    $$\frac{C_2^2 (t^*)^5}{C_1^2} \le \frac{\Delta}{4t^*}=\frac{1}{2}.$$
   
\end{enumerate}
By Bertrand's Postulate,
we can choose a prime $p$ satisfying
\[
C_1 n \le p \le 2C_1 n,
\]
which exists for all sufficiently large $n$.
Let $\mathcal P=(V,\mathcal L)$ be a finite projective plane of order $p$.
Thus, $|V|=p^2+p+1=\Theta(n^2)$, each line contains exactly $p+1$ vertices, each vertex lies on exactly $p+1$ lines, and any two distinct lines intersect
in exactly one vertex.

Let $M=C_2 n$.
For each line $L\in\mathcal L$, independently choose $M$ pairwise disjoint
vertex subsets 
\[
S_{L,1},\dots,S_{L,M}\subseteq L
\]
uniformly among all $t^*$-subsets of $L$.  
This is feasible since
\[
M t^* =C_2 n t^* \le p+1
\]
for sufficiently large $n$ and $C_1$ is chosen sufficiently large
with respect to $C_2$ and $t^*$.
On each $S_{L,i}$, we embed a copy of $G^*$, denoted by $G_{L,i}^*$.
Let $\mathcal{F}_{\text{total}}=\{G_{L,i}^*\mid L\in \mathcal{L}, i\in \{1,2,\ldots ,M\}\}$ denote the collection of all such embedded copies of $G^*$.

Let $R_{0}$ be the resulting graph.
Its vertex set is that of the projective plane, containing $p^2+p+1$ vertices. The edge set is the union of all edges belonging to these embedded copies.
Specifically, if $u, v$ are adjacent in $R_{0}$, there must exist a line $L$ such that $\{u, v\} \subset L$ and $\{u, v\}\in V(G_{L,i}^*)$. 
Furthermore, if two vertices lie on different lines and do not both lie in the same embedded copy of $G^*$, then there is no edge between them.

For a vertex $v \in V$, we say that a copy of $G^*$ \emph{contains} $v$, or equivalently, that $v$ is \emph{covered} by this copy, if $v$ belongs to the vertex set of that copy.
Let $\mathcal{F}_v = \{G_{L,i}^* \mid G_{L,i}^* \in \mathcal{F}_{\text{total}}, v \in V(G_{L,i}^*) \}$ be the set of embedded copies of $G^*$ containing $v$.
Let $\mathcal{L}_v = \{L_1, L_2, \dots, L_{p+1}\}$ denote the set of the $p+1$ distinct lines passing through the vertex $v$.
We define the random variable $X$ to be the total number of embedded copies in $\mathcal{F}_{\text{total}}$ that contain $v$. It means that, $X = |\mathcal{F}_v|$.
Then, we decompose $X$ as a sum of random variables contributions from each line
$$X = \sum_{L \in \mathcal{L}_v} X_L,$$
where $X_L$ is the number of copies on line $L$ that contain $v$.

Consider a fixed line $L \in \mathcal{L}_v$. Let $U_L = \bigcup_{i=1}^M S_{L,i}^{*}$ be a subset of $L$ with $|U_L|=Mt^*$.
By our construction, we chose $M$ pairwise disjoint subsets $S_{L,1}, \dots, S_{L,M} \subset L$, each of size $t^*$,
hence it follows that the vertex $v$ can belong to at most one copy on line $L$.
Thus, $X_L$ is an indicator random variable, that is,
$$X_L =
\begin{cases}
1 & \text{if } v \in U_L, \\
0 & \text{otherwise.}
\end{cases}$$
Since the positions are chosen uniformly at random, then the probability that $v$ is covered by $U_L$ is exactly
$$\mathbb{E}[X_L] = 1\cdot\Pr(X_L = 1)+0\cdot\Pr(X_L = 0) = \frac{|U_L|}{|V(L)|} = \frac{M t^*}{p+1}.$$
By the linearity of expectation, we have
$$\mathbb{E}[X] = \sum_{L \in \mathcal{L}_v} \mathbb{E}[X_L] = \sum_{i=1}^{p+1} \frac{M t^*}{p+1}.$$
Thus,
$$\mathbb{E}[|\mathcal{F}_v|] = (p+1) \cdot \frac{M t^*}{p+1} = M t^*.$$

Consider two distinct copies $A, B \in \mathcal{F}_v$. They lie on distinct lines $L_A$ and $L_B$ intersecting at $v$. 
Set $V_A=V(A)\setminus \{v\}$ and $V_B=V(B)\setminus \{v\}$.
We say that two copies $A, B \in \mathcal{F}_v$ are in \emph{conflict}, denoted by $A \sim B$, if there exist $x \in V_A$ and $y \in V_B$ such that $x$ and $y$ are adjacent in $R_{0}$ (see Figure. \ref{fig1}).
Using this relation, we define the \emph{local conflict graph} $\Gamma_v$. 
Its vertex set is $\mathcal{F}_v$. Two vertices $A, B \in \mathcal{F}_v$ are adjacent in $\Gamma_v$ if and only if they are in conflict in $R_{0}$. Note that the degree of a vertex $A$ in $\Gamma_v$ equals the number of copies in $\mathcal{F}_v$ that are in conflict with $A$.

\begin{figure}[!htbp]
\centering
\includegraphics[width=8cm]{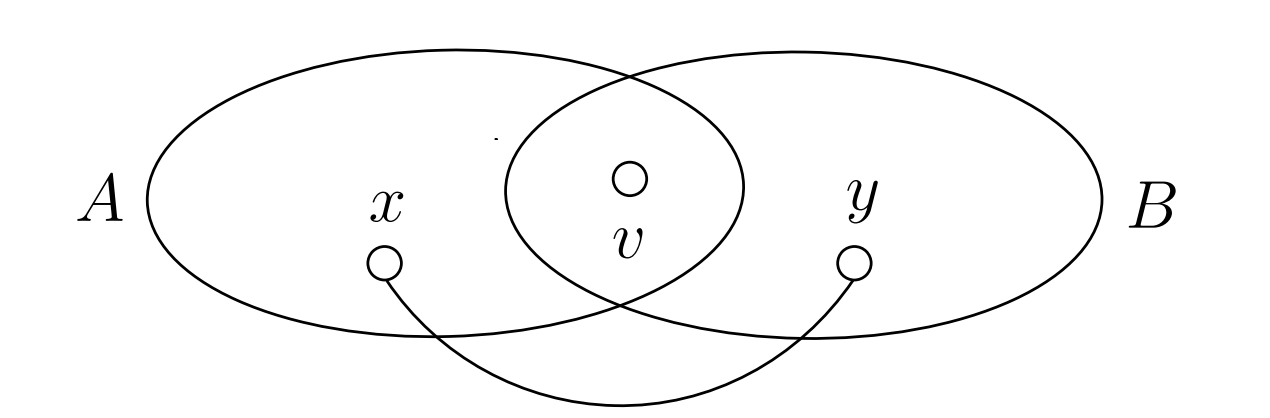}
\caption{Illustration of conflict between copies $A$ and $B$ in $R_{0}$.}\label{fig1}
\end{figure}

Suppose that two copies $A, B\in \mathcal{F}_v$ are in conflict due to $x\in V_A$ and $y\in V_B$.
A pair $\{x, y\}$ determines a unique line $L_{xy}$. An edge exists between any two vertices on $L_{xy}$ if and only if they are covered by the same embedded copy on $L_{xy}$.
Since there are disjoint $M$ copies of $G^*$ on $L_{xy}$, the number of edges joining two vertices on $L_{xy}$ are at most $M\binom{t^*}{2}$.
Thus, the probability $P_{xy}$ of an edge between $x$ and $y$ is bounded by
$$P_{xy} \leq\frac{M \binom{t^*}{2}}{\binom{p+1}{2}} = \frac{M t^*(t^*-1)}{(p+1)p} < \frac{M (t^*)^2}{p^2}.$$
By the Union Bound, the probability that $A$ and $B$ conflict is
$$\Pr(A \sim B)\le \Pr\left(\bigcup_{x \in V_A}\bigcup_{y \in V_B} (x \sim y)\right) \le \sum_{x \in V_A} \sum_{y \in V_B} P_{xy} = (t^*-1)^2 P_{xy} < \frac{M (t^*)^4}{p^2}.$$
Substituting $M = C_2 n$ and $p \ge C_1 n$, we have
$$\Pr(A \sim B) < \frac{C_2 n (t^*)^4}{(C_1 n)^2} = \frac{C_2 (t^*)^4}{C_1^2 n}.$$
Thus the expected degree of a fixed vertex $A$ in $\Gamma_v$ is
$$\mathbb{E}[d_{\Gamma_v}(A)] \le (M t^*) \cdot \Pr(A \sim B) < (C_2 n t^*) \cdot \frac{C_2 (t^*)^4}{C_1^2 n} = \frac{C_2^2 (t^*)^5}{C_1^2}.$$
By our choice of $C_1$, we have
$$\mathbb{E}[d_{\Gamma_v}(A)] \le \frac{\Delta}{4t^*}.$$

A copy \(A \in \mathcal{F}_{\text{total}}\) is said to be \emph{good} provided that \(d_{\Gamma_v}(A) \le \Delta\) for all vertices \(v \in V(A)\). Otherwise, that is, if \(d_{\Gamma_v}(A) > \Delta\) for some vertex \(v \in V(A)\), it is called \emph{bad}.

By Markov's inequality, for a fixed $v \in V(A)$, we know that $$\Pr(d_{\Gamma_v}(A) > \Delta) \le \frac{\mathbb{E}[d_{\Gamma_v}(A)]}{\Delta} \le \frac{1}{4t^*}.$$By the union bound over all $v \in V(A)$,
$$\Pr(A \text{ is bad}) \le \sum_{v \in V(A)} \Pr(d_{\Gamma_v}(A) > \Delta) \le t^* \cdot \frac{1}{4t^*} = \frac{1}{4}.$$
We define $\mathcal{F}'$ as the set of copies in $\mathcal{F}_{\text{total}}$ that are good.
The expected size of $\mathcal{F}'$ is
$$\mathbb{E}[|\mathcal{F}'|] \ge \left(1 - \frac{1}{4}\right) |\mathcal{F}_{\text{total}}| = \frac{3}{4} M (p^2+p+1).$$

We fix a realization of $\mathcal{F}'$ that consists of at least $\frac{3}{4} M (p^2+p+1)$ good copies.
Let $R$ be the graph formed by taking the union of the edge sets of all copies in $\mathcal{F}'$, and its vertex set is also that of the projective plane. By the construction, $R$ is a subgraph of $R_0$ with $V(R)=V(R_0)=V$ and $E(R)\subset E(R_0)$.
For every $v$, let $\Gamma'_v$ be the subgraph of $\Gamma_v$ induced by $\mathcal{F}' \cap \mathcal{F}_v$. By definition, the maximum degree $\Delta(\Gamma'_v) \le \Delta$.

Let $\chi: E(R) \to \{\text{red, blue}\}$ be an arbitrary edge coloring. Assume there is no red induced $G$.
Since $G^* \xrightarrow{\mathrm{ind}} (G, K_3)$, it follows that every copy $A \in \mathcal{F}'$ contains a blue triangle.
Let $\mathcal{T}$ be the collection of these blue triangles. The total number of such triangles is $|\mathcal{T}| = |\mathcal{F}'|$.
For any vertex $v \in V$, let $\text{deg}_{\mathcal{T}}(v)$ denote the number of blue triangles containing $v$.

Consider the set of incidence pairs defined by $I = \{(v, T)\mid v \in V(T),(v, T)\in V \times \mathcal{T}\}$. We compute the size of $I$ in two ways.
First, summing over the vertices, for each fixed $v \in V$, the number of pairs containing $v$ is exactly the number of triangles containing $v$, which is $\text{deg}_{\mathcal{T}}(v)$. Thus,
\[
|I| = \sum_{v \in V} |\{T\mid v \in V(T)\},T \in \mathcal{T}| = \sum_{v \in V} \text{deg}_{\mathcal{T}}(v).
\]
Second, summing over the triangles, for each fixed $T \in \mathcal{T}$, the number of pairs containing $T$ is the number of vertices in $T$, which is $|V(T)| = 3$. Thus,
\[
|I| = \sum_{T \in \mathcal{T}} |\{v\mid v \in V(T),v \in V\}| = \sum_{T \in \mathcal{T}} |V(T)| = 3 |\mathcal{F}'|.
\]
Equating the two expressions for $|I|$, we obtain that
\[
\sum_{v \in V} \text{deg}_{\mathcal{T}}(v) = 3 |\mathcal{F}'|.
\]
which implies
$$\sum_{v \in V} \text{deg}_{\mathcal{T}}(v) =3 |\mathcal{F}'| \ge 3 \cdot \frac{3}{4} M (p^2+p+1) = \frac{9}{4} M (p^2+p+1).$$

By averaging, there exists a vertex $v_0$ 
contained in at least 
$$\text{deg}_{\mathcal{T}}(v_0) \ge \frac{9}{4} M=\frac{9}{4}C_2n$$
blue triangles.
Let $k=\text{deg}_{\mathcal{T}}(v_0)$ and $T_1, \dots, T_k$ be the blue triangles containing $v_0$.
Since $T_1, T_2, \dots ,T_k$ are contained in pairwise distinct copies of $G^*$, by the property of the projective plane, $T_1, T_2, \dots , T_k$ lie on pairwise distinct lines in $\mathcal{L}_v$.
In particular, $V(T_i)\cap V(T_j)=\{v_0\}$ for any distinct $i$ and $j$. 

For each triangle $T_i \in \mathcal{T}_{v_0}$, let $A_i \in \mathcal{F}'$ be the unique copy of $G^*$ containing $T_i$.
Let $\mathcal{T}_{v_0} = \{T_1, \dots, T_k\}$ and $S = \{A_1, \dots, A_k\} \subseteq \mathcal{F}'$.

Consider $\Gamma'_{v_0}[S]$, which is the subgraph of
the local conflict graph $\Gamma'_{v_0}$ induced by $S$. 
Since the maximum degree of $\Gamma'_{v_0}[S]$ is at most $\Delta$, by Lemma \ref{Caro–Wei}, we conclude that the independence number satisfies
$$\alpha(\Gamma'_{v_0}[S]) \ge \frac{k}{\Delta + 1} \ge \frac{9 C_2 n}{4(\Delta + 1)}.$$
Using the condition $\frac{9 C_2 }{4(\Delta + 1)} \ge 1$, we have
$$\alpha(\Gamma'_{v_0}[S]) \ge n.$$
Take an independent set $\{A_{i_1}, \dots, A_{i_n}\}$ of $\Gamma'_{v_0}[S]$.
Since these copies are independent in $\Gamma'_{v_0}$, there are no edges in $R$ between $A_{i_u} \setminus \{v_0\}$ and $A_{i_w} \setminus \{v_0\}$ for any $u \neq w$.
Consequently, the corresponding blue triangles $T_{i_1}, \dots, T_{i_n}$ contained in these copies form a blue induced $F_n$.

We have shown that for sufficiently large $n$, there exists a graph $R$ on $\Theta(n^2)$ vertices such that any red-blue coloring of $R$ contains either a red induced $G$ or a blue induced $F_n$. Hence, $r_{\mathrm{ind}}(G,F_n)\leq Cn^2$.
\end{proof}

\section{Fans versus stars}




We denote the size (number of edges) of the maximum matching in $G$ as $\alpha^{\prime}(G)$; for a maximum matching $M$, its vertex set (all vertices incident to edges in $M$ ) is written as $V(M)$, and $|V(M)|=2|M|$ (each edge contributes 2 vertices).

\setcounter{claim}{0}
\setcounter{case}{0}

\begin{lemma}\label{th-K_{1,l}}
Let $\ell$ and $n$ be two positive integers with $\ell\leq n$. Then
$$
r_{\mathrm{ind}}(K_{1,\ell},F_{n})\geq2n+2\ell-1.
$$
\end{lemma}

\begin{proof}
Let $G$ be an arbitrary graph of order $2n+2\ell-2$, and we shall show that $G$ has a red-blue coloring without a red induced $K_{1,\ell}$ and a blue induced $F_n$.
Let $\alpha'(G)=k$ and $M=\{x_iy_i\mid 1\leq i\leq k\}$ be a maximum matching of $G$.
Set $I=V(G)\setminus V(M)$, $X=\{x_i\mid 1\leq i\leq k\}$, and $Y=\{y_i\mid 1\leq i\leq k\}$.
Note that $k\leq n+\ell-1<2n$.

By the maximality of $M$, $I$ is independent.
Furthermore, if $x_i$ has two neighbors $u,v$ $I$ and $y_i$ has a neighbor $w$ in $I$, then either $ux_iy_iw$ or $vx_iy_iw$ is an augmenting path with respect to $M$, a contradiction by the maximality of $M$.
Thus, without loss of generality, we may assume that $e_G(y,I)\leq 1$ for any $y\in Y$.
We set $f(x)=\max\{e_G(x,I)-2n+k, 0\}$.
If $f(x)\geq \ell-1$, then we have 
$$
|V(G)|\geq e_G(x,I)+2k\geq f(x)+(2n-k)+2k\geq 2n+k+\ell-1\geq 3n+\ell-1>2n+2\ell-2,
$$
a contradiction.
Thus we have $f(x)\leq\ell-2$ for each $x\in X$.
    
\begin{claim}\label{claim:flow}
    There exists a subgraph $F$ of $G$ such that $V(F)=X\cup I$, $d_F(x)=f(x)$ for any $x\in X$, and $d_F(v)\leq \ell-1$ for any $v\in I$.
\end{claim}

\begin{proof}
    We use the Max-Flow Min-Cut theorem.
    Let $H$ be an auxiliary graph defined by
    \[V(H)=X\cup Y\cup \{s,t\}, \quad E(H)=E_G(X,I)\cup \{sx\mid x\in X\}\cup \{vt\mid v\in I\}\]
    and let $c$ be a capacity function defined by $c(xv)=1$ for each $xv\in E_G(X,Y)$, $c(sx)=f(x)$ for each $x\in X$, and $c(vt)=\ell-1$ for each $v\in I$.
    Then the existence of an $s-t$ flow of amount $\sum_{x\in X}f(x)$ implies the existence of a desired subgraph $F$ of $G$.

    It suffices to show that for any $S\subseteq X$ and $T\subseteq I$, the cut condition
    \begin{align*}
        \sum_{x\in X}f(x)\leq \sum_{x\in X\setminus S}f(x)+e_H(S,I\setminus T)+|T|(\ell-1),
    \end{align*}
    equivalently
    \begin{align}
        \sum_{x\in S}f(x)\leq e_H(S,I\setminus T)+|T|(\ell-1)\label{eq1}
    \end{align}
    is satisfied.
    Since 
    \begin{align*}
        \sum_{x\in S}f(x)\leq \sum_{x\in S}(e_G(x,I)-2n+k)=e_G(S,I)-|S|(2n-k)=e_G(S,I\setminus T)+e_G(S,T)-|S|(2n-k)
    \end{align*}
    and $e_G(S,T)\leq |S||T|$,
    we substitute $\sum_{x\in S}f(x)=e_G(S,I\setminus T)+|S||T|-|S|(2n-k)$
    to have
    \begin{align}
        |S||T|-|S|(2n-k)\leq |T|(\ell-1),\label{eq2}
    \end{align}
    equivalently
    \begin{align}
        \bigl(|S|-(\ell-1)\bigr)\bigl(|T|-(2n-k)\bigr)\leq (\ell-1)(2n-k),\label{eq3}
    \end{align}
    which implies (\ref{eq1}).
    As the right side of (\ref{eq3}) is non-negative and the inequality holds when $|S|\leq \ell-1$ and $|T|\leq 2n-k$, it suffices to consider the case $|S|\geq \ell-1$ and $|T|\geq 2n-k$.
    Then the left side is monotonically increasing with respect to both $|S|$ and $|T|$, and thus the tightest constraint occurs when both are maximized, i.e., $|S|=k$ and $|T|=|I|$.
    Substituting these and $2n=|I|+2k-2\ell+2$ to (\ref{eq2}), we have
    \begin{align*}
        k|I|-k(|I|+k-2\ell+2)\leq |I|(\ell-1),
    \end{align*}
    and hence 
    \begin{align*}
        k(2\ell-k-2)\leq |I|(\ell-1).
    \end{align*}
    As $|I|=2n+2\ell-2-2k\geq 4\ell-2-2k$, we can substitute $|I|=4\ell-2-2k$ to have
    \begin{align*}
        k(2\ell-k-2)\leq (4\ell-2-2k)(\ell-1).
    \end{align*}
    Rearrangement of the terms yields a quadratic in terms of $k$;
    \begin{align}
        k^2-4(\ell-1)k+4\ell^2-6\ell+2\geq 0.\label{eq4}
    \end{align}
    The discriminant of this quadratic is $\Delta=-8(\ell-1)\geq 0$ for $\ell\geq 1$.
    Thus, (\ref{eq4}) holds any $\ell\geq 1$ and any $k$, and hence (\ref{eq1}) holds for any $S\subseteq X$ and $T\subseteq I$.
    This completes the proof of Claim~\ref{claim:flow}.
\end{proof}
We color $M\cup E(F)$ in red and the other edges in blue.
By the definition of $F$, every vertex $y\in Y$ is incident with one red edge, every vertex $v\in I$ incident with at most $\ell-1$ red edges, and every vertex $x\in X$ is incident with $f(x)+1\leq \ell-1$ red edges.
Thus no red induced $K_{1,\ell}$ exists.
    
Suppose that there exists a blue induced $F_n$ with vertices $U$ and the center $u$.
Since at most one of $\{x_i,y_i\}$ can be in $U$ for each $x_iy_i\in M$, we have $k\geq |U\cap (X\cup Y)|$.
If $u\in I$, then since $I$ is an independent set, it follows that $k\geq |U\setminus \{u\}|=2n$, and hence $|V(G)|\geq 2k+|U\cap I|=2k+1\geq 4n+1>2n+2\ell-2$, a contradiction.
If $u\in Y$, then since $u$ has at most one neighbor in $I$, it follows that $|U\cap (X\cup Y)|\geq (2n+1)-1=2n$, and hence $|V(G)|\geq 2k+1\geq 4n+1$, again a contradiction.
Thus, we infer that $u\in X$.
Since $u$ has $f(u)$ neighbors in $I$, which are joined by red edges to $u$, it follows that
$$
|U\cap I|\leq e_G(u,I)-f(u)\leq e_G(u,I)-(e_G(u,I)-2n+k)=2n-k.
$$
This implies that $|U|=|U\cap (X\cup Y)|+|U\cap I|\leq k+(2n-k)=2n$, a contradiction.

Combining these, we conclude that $G$ has no blue induced $F_n$, and hence this is a desired coloring of $G$.
Therefore,
\[
r_{\mathrm{ind}}(K_{1,\ell}, F_n) \ge 2n + 2\ell - 1.
\]
\end{proof}

\begin{lemma}\label{Star-lower}
For any positive integers $\ell$ and $n$,
    \[r_{\mathrm{ind}}(K_{1,\ell},F_n)\leq r_{\mathrm{ind}}(K_{1,\ell},(\ell+n-1)K_2)+1\leq (\ell+n-1)(\ell+1)+1. \]
\end{lemma}

\begin{proof}
We first show the first inequality.
    We fix positive integers $\ell$ and $n$.
    Let $H$ be a graph of order $r_{\mathrm{ind}}(K_{1,\ell}, (\ell+n-1)K_2)$ such that $H\overset{ind}{\longrightarrow} (K_{1,\ell}, (\ell+n-1)K_2)$.
    Let $G$ be a graph obtained from $H$ by adding a new vertex $w$ which is adjacent to all the vertices of $H$.
    Now we color the edges of $G$ with two colors red and blue, and assume that there is no red induced $K_{1,\ell}$.
    Since $H\overset{ind}{\longrightarrow} (K_{1,\ell}, (\ell+n-1)K_2)$, $G-w$ contains a blue induced matching $\{u_iv_i\mid 1\leq i\leq \ell+n-1\}$ of size $\ell+n-1$.
    As $G$ has no red induced $K_{1,\ell}$, at most $\ell-1$ edges in $\{u_iv_i\mid 1\leq i\leq \ell+n-1\}$ can have end vertices that are adjacent to $w$ with red edges.
    As $\ell+n-1-(\ell-1)=n$, without loss of generality, we may assume that both $u_iw$ and $v_iw$ are blue edges for each $i\in \{1,\dots ,n\}$, and hence $\{w\}\cup \bigcup_{i=1}^{n}\{u_i,v_i\}$ induces a blue $F_n$.

    For the second inequality, we consider $(\ell+n-1)K_{1,\ell}$, the union of $\ell+n-1$ disjoint copies of $K_{1,\ell}$.
    It is easy to see that $(\ell+n-1)K_{1,\ell}$ has order $(\ell+n-1)(\ell+1)$ and that $(\ell+n-1)K_{1,\ell}\overset{ind}{\longrightarrow}(K_{1,\ell},(\ell+n-1)K_2)$.
\end{proof}

\setcounter{case}{0}
\setcounter{claim}{0}

\begin{lemma}\label{le-K_{1,2}}
For each positive integer $n\geq2$, we have 
\[
r_{\mathrm{ind}}(K_{1,2}, F_n)=3n + 4.
\]
\end{lemma}

\begin{proof}
By Lemma \ref{Star-lower} with $\ell=2$, we know that $r_{\text {ind}}\left(K_{1,2}, F_n\right) \leq 3 n+4$.
We shall show that any graph $G$ with at most $3 n+3$ vertices has an edge coloring in two colors red and blue that contains neither a red induced $K_{1,2}$ nor a blue induced $F_n$. Let $r$ be the chromatic number of the complement $\overline{G}$ of $G$. Let $\phi$ be a proper $r$-coloring of $\overline{G}$ such that the number of color classes of size one is as large as possible. Let $V_1, \ldots, V_r$ be the color classes of $\phi$. Then each $V_i$ is a clique of $G$ for $i \in\{1, \ldots, r\}$.

We define a 2-edge coloring $\chi_1$ of $G$ so that the edges in $G\left[V_i\right]$ are red for each $i \in\{1, \ldots, r\}$ and the edges between $V_i$ and $V_j$ are blue for any $i$ and $j$ with $1 \leq i<j \leq r$. Since each red edge lies in a clique of $G$, it follows that $G$ has no induced red copy of $K_{1,2}$. If $G$ has no induced blue copy of $F_n$, then $\chi_1$ is a desired coloring. Thus, we assume that $G$ has an induced blue copy of $F_n$ induced by $A:=\{a_1,\dots ,a_{2n+1}\}$, where $a_1a_{2s}a_{2s+1}$ is a blue triangle for each $s\in \{1,\dots ,n\}$. As two vertices in the same class $V_i$ are joined by a red edge,
\begin{quote}
    (a) vertices of $A$ are in the pairwise distinct classes in $V_1, \ldots, V_r$,    
\end{quote}
which forces $r \geq\left|V\left(F_n\right)\right|=2 n+1$. Without loss of generality, we may assume that $a_i \in V_i$ for each $i \in\{1, \ldots, 2 n+1\}$. Let $A_1:=\left\{a_i\mid 1 \leq i \leq 2 n+1,|V_i|=1\right\}$, $A_2:=\left\{a_i\mid 1 \leq i \leq 2 n+1,|V_i|=2\right\}$, and $A_3:=\left\{a_i\mid 1 \leq i \leq 2 n+1, |V_i| \geq 3\right\}$. Since $V_1 \cup \cdots \cup V_r$ is defined from an $r$-coloring $\phi$ of $\overline{G}$, $E_{\overline{G}}\left(V_i, V_j\right) \neq \emptyset$ for any distinct $V_i$ and $V_j$, and hence
\begin{quote}
    (b) $A_1$ is an independent set of $G$.  
\end{quote}

Furthermore, for a triangle $a_1 a_{2 s} a_{2 s+1}$ of the induced $F_n$, if exactly one of $\left\{a_1, a_{2 s}, a_{2 s+1}\right\}$, say $a_1$, belongs to $A_1$ and other two vertices belong to $A_2$, then we let $V_1^{\prime}:=\left\{a_1, a_{2 s}, a_{2 s+1}\right\}, V_{2 s}^{\prime}:=V_{2 s} \setminus\left\{a_{2 s}\right\}$, and $V_{2 s+1}^{\prime}:=V_{2 s+1} \setminus\left\{a_{2 s+1}\right\}$. Since each of $V_1^{\prime}, V_{2 s}^{\prime}, V_{2 s+1}^{\prime}$ is a clique of $G$ and $\left|V_{2 s}^{\prime}\right|=\left|V_{2 s+1}^{\prime}\right|=1,\left(\left\{V_2, V_3, \ldots, V_{2 n+1}\right\} \setminus\left\{V_{2 s}, V_{2 s+1}\right\}\right) \cup\left\{V_1^{\prime}, V_{2 s}^{\prime}, V_{2 s+1}^{\prime}\right\}$ induces a proper $r$-coloring $\phi^{\prime}$ of $\overline{G}$ with a strictly larger number of classes of size one than $\phi$, a contradiction by the choice of $\phi$. 
Thus, 
\begin{quote}
    (c) For each $s\in \{1,\ldots ,n\}$, if one of $\{a_1, a_{2s}, a_{2s+1}\}$ belongs to $A_1$, then at least one of the others belongs to $A_3$.
\end{quote}
If $a_1\in A_1$, then (b) and (c) imply that $|V_{2s}|+|V_{2s+1}|\geq 2+3=5$ for any $s\in \{1,\ldots ,n\}$, and hence $|V(G)|=\sum_{i=1}^r|V_i|\geq \sum_{i=1}^{2n+1}|V_i|=|V_1|+\sum_{s=1}^n(|V_{2s}|+|V_{2s+1}|) \geq 1+5n>3n+3$, a contradiction.
If $a_1\in A_2$, then (b) and (c) imply that $|V_{2s}|+|V_{2s+1}|\geq 4$ for any $s\in \{1,\ldots ,n\}$, and hence $|V(G)|\geq \sum_{i=1}^{2n+1}|V_i| =|V_1|+\sum_{s=1}^n(|V_{2s}|+|V_{2s+1}|)\geq 2+4n>3n+3$, again a contradiction.
Thus, we conclude that $a_1\in A_3$.
Since $|V_{2s}|+|V_{2s+1}|\geq 3$ for any $s\in \{1,\dots ,n\}$ by (b) and (c), we have that 
\[
|V(G)|\geq \sum_{i=1}^{2 n+1}\left|V_i\right|=|V_1|+\sum_{s=1}^n(|V_{2s}|+|V_{2s+1}|) \geq 3+3n,
\]
and hence $|V(G)|=3 n+3$. 
In particular, since the above inequalities are indeed equalities, without loss of generality, we may assume that $\left|V_1\right|=3$, $\left|V_{2 s}\right|=1$, $\left|V_{2 s+1}\right|=2$ for every $s \in\{1, \ldots, n\}$, and $V(G)=\bigcup_{i=1}^{2 n+1} V_i$. 
Let $V_{2s}=\left\{a_{2s}\right\}$ and $V_{2 s+1}=\left\{a_{2 s+1}, b_{2 s+1}\right\}$ for each $s \in\{1, \ldots, n\}$, and let $V_1 = \{a_1, a_{2n+2}, a_{2n+3}\}$. Then we have $A_1=\{a_{2s}\mid 1\leq s\leq n\}$, $A_2=\{a_{2s+1}\mid 1\leq s\leq n\}$, and $A_3=\{a_1\}$.
Note that the fact $E_{\overline{G}}(V_{2s},V_{2s+1})\neq \emptyset$ implies that $a_{2s}b_{2s+1}\notin E(G)$ for each $s\in \{1,\dots ,n\}$. Set $B=\{b_{2s+1}\mid 1\leq s\leq n\}$.

\begin{claim}\label{claim1-exact}
$N_G\left(a_1\right)=V(G) \setminus\left\{a_1\right\}$.    
\end{claim}
\begin{proof}
By symmetry, it suffices to show that $b_3 a_1 \in E(G)$. Let $W_1=\left\{a_1, a_2, a_3\right\}$, $W_2=\left\{a_3\right\}$, $W_3=\left\{a_1, b_1\right\}$, and $W_i=V_i$ for each $i \in\{4, \ldots, 2 n+1\}$. We consider another 2-edge coloring $\chi_{2}$ of $G$ by coloring all the edges in $G\left[W_i\right]$ red and other edges blue. Since each $W_i$ is a clique of $G$, it follows that $G$ does not have an induced red $K_{1,2}$ with respect to $\chi_{2}$. If $G$ does not have an induced blue $F_n$, then $\chi_{2}$ is a desired coloring. Thus, we assume that there exists $w_i \in W_i$ for each $i \in\{1, \ldots, 2n+1\}$ such that $\left\{w_1, \ldots, w_{2 n+1}\right\}$ induces a blue $F_n$. Obviously, we have $w_i=a_i$ for $i \in\{2 s \mid 2 \leq s \leq n\}$ and $w_2=b_3$. By the similar argument to that for $\chi_1$, we know that the class of size 3 corresponds to the center of the $F_n$, and hence $w_1$ is the center. Since $W_1 \cap \bigcap_{s=2}^n N_G\left(w_{2 s}\right)=\left\{a_1, a_2, a_3\right\} \cap \bigcap_{s=2}^n N_G\left(a_{2 s}\right)=\left\{a_1\right\}$, it follows that $w_1=a_1$. Therefore, the fact that $w_1 \in N_G\left(w_2\right)$ implies $a_1 \in N_G\left(b_3\right)$.    
\end{proof}

\begin{claim}\label{claim2-exact}
$G-a_1$ does not contain $F_n$ as an induced subgraph.   
\end{claim}
\begin{proof}
Suppose that $\left\{w_1, \ldots, w_{2 n+1}\right\} \subseteq V(G) \setminus\left\{a_1\right\}$ induces a copy of $F_n$, where $w_1 w_{2 s} w_{2 s+1}$ is a triangle for each $s \in\{1, \ldots, n\}$. Then we consider a vertex partition $W_1 \cup \cdots \cup W_t$ such that $W_1=\left\{a_1, w_2, w_3\right\}, W_2=\left\{w_1, w_4, w_5\right\}$, $W_i=\left\{w_{2 i}, w_{2 i+1}\right\}$ for $i \in\{3, \ldots, n\}$, and $W_i$ consists of a single vertex for each $i \in\{n+1, \ldots, t\}$. Then we have $3 n+3=|V(G)|=\sum_{i=1}^t\left|W_i\right|=6+2(n-2)+(t-n)$, which implies $t=2 n+1=r$. Since each $W_i$ is a clique of $G$, this partition induces a proper $r$-coloring of $\overline{G}$ with $n+1$ color classes of size one, a contradiction by the choice of $\phi$.
\end{proof}

\begin{claim}\label{claim:max matching}
        The matching number of $G-a_1$ is at most $n+1$.
        In particular, $B$ is an independent set of $G$ and $E_G(A_1,B)=\emptyset$.
    \end{claim}

    \begin{proof}
        Assume to the contrary that $G-a_1$ has a matching $M=\{x_jy_j\mid 1\leq j\leq n+2\}$ of size $n+2$.
        We consider a vertex partition $W_1\cup \dots \cup W_t$ of $G$ such that $W_1=\{a_1,x_1,y_1\}$, $W_i=\{x_i,y_i\}$ for $i\in \{2,\dots ,n+2\}$, and $W_i$ consists of a single vertex for each $i\in \{n+3,\dots ,t\}$.
        Then we have $3n+3=|V(G)|=\sum_{i=1}^{t}|W_i|=3+2(n+1)+(t-n-2)$, which implies that $t=2n<r$.
        Since each $W_i$ is a clique of $G$, this partition induces a proper $t$-coloring of $\overline{G}$ with $t < r$, contradicting the definition of $r$. Thus, the matching number of $G-a_1$ is at most $n+1$.

        For the latter statement, we consider a matching $M^*=\{a_{2n+2}a_{2n+3}\} \cup \{a_{2s}a_{2s+1} \mid 1 \le s \le n\}$ of size $n+1$ in $G-a_1$.
        Since the maximum matching size is exactly $n+1$ and $V(B) \cap V(M^*) = \emptyset$, it follows that $B$ must be an independent set of $G$; otherwise, any edge within $B$ could be added to $M^*$ to form a matching of size $n+2$, a contradiction.
        
        Next, we show that $E_G(A_1, B) = \emptyset$. Suppose for a contradiction that there is an edge $a_{2s_1} b_{2s_2+1} \in E(G)$ for some $s_1, s_2 \in \{1, \dots, n\}$. If $s_1 = s_2$, this contradicts the earlier deduction that $a_{2s}b_{2s+1} \notin E(G)$, which follows from $E_{\overline{G}}(V_{2s}, V_{2s+1}) \neq \emptyset$. If $s_1 \neq s_2$, since the vertices $a_{2s_1+1}$ and $b_{2s_1+1}$ belong to the same clique $V_{2s_1+1}$, it follows that $a_{2s_1+1}b_{2s_1+1} \in E(G)$, and hence, we can construct a larger matching $(M^* \setminus \{a_{2s_1}a_{2s_1+1}\}) \cup \{a_{2s_1}b_{2s_2+1}, a_{2s_1+1}b_{2s_1+1}\}$. This new matching has size $(n+1) - 1 + 2 = n+2$, which again contradicts the maximum matching size. Therefore, $E_G(A_1, B) = \emptyset$.
    \end{proof}

    \begin{claim}\label{claim:a2 b}
        $E_G(A_2,B)=\{a_{2s+1}b_{2s+1}\mid 1\leq s\leq n\}$, and $A_2$ is an independent set of $G$.
    \end{claim}

    \begin{proof}
        We consider an edge coloring $\chi_3$ such that the edges $\{a_{2s}a_{2s+1}\mid 1\leq s\leq n+1\}\cup \{a_1a_{2n+2}, a_1a_{2n+3}\}$ are colored red, and others are colored blue.
        Since $\{a_{2s}a_{2s+1} \mid 1 \le s \le n\}$ is a matching disjoint from $\{a_{2n+2}, a_{2n+3}\}$, and $a_1 a_{2n+2} a_{2n+3}$ forms a triangle in $G$, $G$ contains no induced red $K_{1,2}$ under $\chi_3$. If there is no induced blue $F_n$, then $\chi_3$ is a desired coloring, and the theorem holds.
        Suppose that there is a blue induced $F_n$ under $\chi_3$.
        By Claim~\ref{claim2-exact}, we know that the 
        induced $F_n$ must contain the vertex $a_1$. Furthermore, the edges between $a_1$ and both $a_{2n+2}$ and $a_{2n+3}$ are red, so $a_{2n+2}$ and $a_{2n+3}$ cannot be part of this induced $F_n$. Also, we know the center of this blue induced $F_n$ must lie in $V_1$. Consequently, the center must be $a_1$. Thus, the $2n$ leaf vertices of this $F_n$ must be chosen from $A_1 \cup A_2 \cup B$, and they must induce a blue $nK_2$.
        
        Because the edge $a_{2s}a_{2s+1}$ is red for each $s \in \{1, \dots, n\}$, the induced blue $nK_2$ can contain at most one vertex from each pair $\{a_{2s}, a_{2s+1}\}$. This implies that the blue $nK_2$ contains at most $n$ vertices from $A_1 \cup A_2$. By Claim \ref{claim:max matching}, $B$ is an independent set and $E_G(A_1, B) = \emptyset$. Since every vertex in the induced blue $nK_2$ must have degree exactly 1, the $n$ vertices of $B$ must be matched with vertices in $A_2$. This forces the remaining $n$ vertices of the $F_n$ to be exactly the set $A_2$.

        Hence, $A_2 \cup B$ induces the blue $nK_2$. This implies that in the original graph $G$, the bipartite graph between $A_2$ and $B$ must be a perfect matching, and there are no other edges within $A_2$ or between $A_2$ and $B$. Since $a_{2s+1}b_{2s+1} \in E(G)$ for each $s$, these $n$ edges perfectly form the required matching. Therefore, $E_G(A_2, B) = \{a_{2s+1}b_{2s+1} \mid 1 \le s \le n\}$ and $A_2$ is an independent set.
\end{proof}

By Claims~\ref{claim:max matching} and \ref{claim:a2 b}, $E(G-\{a_1,a_{2n+2},a_{2n+3}\})=\bigcup_{a=1}^{n} \{a_{2s}a_{2s+1}, a_{2s+1}b_{2s+1}\}$.
In particular $a_{2s}$ and $b_{2s+1}$ are symmetric.
In the following proof, we divide the case according to the edge between $\{a_{2n+2}, a_{2n+3}\}$ and $A_1\cup A_2\cup B$.

    \medskip

\begin{case}
    $E_G(\{a_{2n+2},a_{2n+3}\},A_1\cup A_2\cup B)=\emptyset$.
\end{case}

    In this case, $E(G-a_1)=\{a_{2n+2}a_{2n+3}\}\cup \bigcup_{s=1}^n \{a_{2s}a_{2s+1}, a_{2s+1}b_{2s+1}\}$.
    Let $\chi_4$ be an edge coloring of $G$ such that $\{a_{2s}a_{2s+1}\mid 1\leq s\leq n+1\}\cup \{a_1a_2,a_1a_3\}$ are colored red and others are blue.
    Since $\{a_{2s}a_{2s+1}\mid 1\leq s\leq n+1\}$ is a matching and $a_1$ is adjacent to every other vertex, it follows that $G$ has no red induced $K_{1,2}$ under $\chi_4$.
    On the other hand, $G$ contains a unique blue induced matching $\{a_{2s+1}b_{2s+1}\mid 1\leq s\leq n\}$ of size $n$ but $a_1a_3$ is red, which implies that $G$ has no blue induced $F_n$.
    Thus $\chi_4$ is a desired coloring.

  \begin{case}
      $E_G(\{a_{2n+2},a_{2n+3}\},A_1\cup A_2\cup B)\neq \emptyset$.
  \end{case}

    By symmetry, we may assume that there is a vertex $x\in \{a_3,b_3\}\cap N_G(a_{2n+2})$.
    Let $\chi_5$ be an edge coloring of $G$ such that $\{a_{2s}a_{2s+1}\mid 1\leq s\leq n+1\}\cup \{a_1a_{2n+2}, a_1x\}$ are colored red and others are blue.
    Since $\{a_{2s}a_{2s+1}\mid 1\leq s\leq n+1\}$ is a matching and $a_1$ is adjacent to every other vertex, it follows that $G$ has no red induced $K_{1,2}$ in $\chi_5$.
    On the other hand, $G$ contains a unique blue induced matching $\{a_{2s+1}b_{2s+1}\mid 1\leq s\leq n\}$ of size $n$ but $a_1b_3$ is red, which implies that $G$ has no blue induced $F_n$.
    Thus $\chi_5$ is a desired coloring.
\end{proof}

\newtheorem*{mainthm2}{\rm\bf Theorem~\ref{star-fan-bound}}
\begin{mainthm2}[Restated]
Let $n$, $\ell$ be positive integers with $\ell\leq n$. Then
$$
2n+2\ell-1\leq r_{\mathrm{ind}}(K_{1,\ell},F_{n})\leq (\ell+n-1)(\ell+1)+1.
$$
Furthermore, if $l=2$, we have
$r_{\mathrm{ind}}(K_{1,2},F_{n})=3n+4$.
\end{mainthm2}
\begin{proof}
From Lemmas \ref{th-K_{1,l}}, \ref{Star-lower} and \ref{le-K_{1,2}}, we could obtain the conclusion immediately.
\end{proof}

\section{Conclusions}

In this paper, we explore the asymptotic behavior and exact values of induced Ramsey numbers for fan graphs. For a general fixed graph $G$, we establish a quadratic upper bound $r_{\mathrm{ind}}(G, F_n) \le O(n^2)$. In addition to these robust asymptotic bounds, we investigate the extremely sparse regime by evaluating $r_{\mathrm{ind}}(K_{1,\ell}, F_n)$, obtaining tight linear bounds for general $\ell$ and determining the exact value for $\ell=2$. We conclude with several open problems for future research.

\begin{enumerate}
    \item While we determined the exact value $r_{\mathrm{ind}}(K_{1,2}, F_n) = 3n+4$, there remains a gap between our upper and lower bounds for $r_{\mathrm{ind}}(K_{1,\ell}, F_n)$ when $\ell \ge 3$. Determining the exact exact formula, or at least the precise leading coefficient, for general $\ell$ is an interesting combinatorial challenge.
    
    \item We have shown that $r_{\mathrm{ind}}(G, F_n) = O(n^2)$ for any fixed $G$. An immediate question is whether this quadratic upper bound is asymptotically tight. Does there exist a fixed graph $G$ and a constant $c > 0$ such that $r_{\mathrm{ind}}(G, F_n) \ge c n^2$? Or can the upper bound be further improved to $O(n^{2-\epsilon})$?
    
   \item Our projective plane framework successfully handled the triangle fan $F_n$, which is a cograph. It would be of great interest to investigate whether this methodology can be generalized to establish low-degree polynomial bounds for fans of non-cographs. For instance, if $F_n^{(C_5)}$ denotes a fan consisting of $n$ pentagons sharing a common vertex, does $r_{\mathrm{ind}}(G, F_n^{(C_5)})$ still admit an $O(n^2)$ bound?
\end{enumerate}

\end{document}